\documentclass[12pt,reqno]{amsart}

\usepackage[T1]{fontenc}
\usepackage{newpxtext}                     
\usepackage[varg,bigdelims]{newpxmath}     
\usepackage[mathscr]{eucal}                
\usepackage{bm}
\usepackage{microtype}

\usepackage[a4paper,textwidth=16cm,textheight=22cm]{geometry}
\usepackage{setspace}

\usepackage{amsmath,mathtools}     
\usepackage{graphicx}
\usepackage{xcolor}
\usepackage{ytableau}
\usepackage{stackengine}
\usepackage{enumitem}
\usepackage{tikz-cd}
\usetikzlibrary{arrows.meta}

\usepackage[backend=biber,style=alphabetic,maxbibnames=99, isbn=false, doi=false, url=false]{biblatex}
\usepackage[colorlinks,
            linkcolor={red!55!black},
            citecolor={green!40!black},
            urlcolor={blue!60!black}]{hyperref}
\usepackage[capitalise,noabbrev]{cleveref}

\tikzset{
  bulletepi/.style={
    {Circle[fill=black, scale=0.7]}-{Straight Barb[] Straight Barb[]}
  }
}

\theoremstyle{plain}
\newtheorem{theorem}{Theorem}[section]
\newtheorem{corollary}[theorem]{Corollary}

\newtheorem{lemma}[theorem]{Lemma}
\newtheorem{proposition}[theorem]{Proposition}

\newtheorem*{lemma*}{Technical Lemma}

\theoremstyle{definition}
\newtheorem{defn}[theorem]{Definition}
\newtheorem{example}[theorem]{Example}

\theoremstyle{remark}
\newtheorem{rem}[theorem]{Remark}

\theoremstyle{plain}

\newcommand{\crys}{\mathcal{B}}
\newcommand{\mcdw}{\mathcal{D}_{\!\lambda\mu}^{\nu}(w)}
\newcommand{\mcd}{\mathcal{D}_{\!\lambda\mu}^{\nu}}
\newcommand{\mcmw}[1][w]{\mathcal{M}_{\!\lambda\mu}^{\nu}(#1)}
\newcommand{\posreal}[1][+]{\Delta_{re}^{#1}}
\newcommand{\posimag}{\Delta_{im}^+}

\newcommand{\cato}{\mathcal{O}}
\newcommand{\catoint}{\mathcal{O}^{\mathrm{int}}}
\newcommand{\lie}{\mathfrak}
\newcommand{\twoheadrightarrowtail}{%
  \mathrel{\mathrlap{\scalebox{0.75}{\textbullet}}\twoheadrightarrow}}
\newcommand{\canmor}{\twoheadrightarrowtail}
\newcommand{\lrw}{c_{\lambda\mu}^\nu(w)}
\newcommand{\lr}{c_{\lambda\mu}^\nu}
\newcommand{\hwcrys}[1][]{{b}_{#1}}

\DeclareMathOperator{\ann}{ann}
\DeclareMathOperator{\res}{res}
\DeclareMathOperator{\Res}{Res}
\DeclareMathOperator{\Hom}{Hom}

\begin{document}

\title{Kostant--Kumar modules: presentation and multiplicities}

\author{Manika Gupta}
\address{The Institute of Mathematical Sciences, Chennai, India}
\address{Homi Bhabha National Institute, Training School Complex, Anushakti Nagar, Mumbai 400094, India}
\email{manikag@imsc.res.in}

\author{K.~N. Raghavan}
\address{Krea University, Sri City, Andhra Pradesh, India}
\email{raghavan.komaranapuram@krea.edu.in}

\author{Sankaran Viswanath}
\address{The Institute of Mathematical Sciences, Chennai, India}
\address{Homi Bhabha National Institute, Training School Complex, Anushakti Nagar, Mumbai 400094, India}
\email{svis@imsc.res.in}

\thanks{MG and SV acknowledge partial support through a DAE Apex grant to the Institute of Mathematical Sciences.}

\begin{abstract}
Kostant--Kumar modules $K(\lambda,w,\mu)$ are submodules of a tensor
product $V(\lambda)\otimes V(\mu)$ of irreducible highest weight modules
over a symmetrizable Kac--Moody algebra, indexed by Weyl group elements
$w$; their decomposition numbers $c^\nu_{\lambda\mu}(w)$ refine ordinary
tensor product multiplicities. We study them module-theoretically. We show
that $c^\nu_{\lambda\mu}(w)$ is computed by a natural quotient of the
Kostant--Parthasarathy--Ranga Rao--Varadarajan multiplicity space, via
orthogonal projection onto a Demazure module. For $\mathfrak{g}$
finite-dimensional semisimple or symmetric Kac--Moody, we present
$K(\lambda,w,\mu)$ by generators and relations, extending the presentation
of Demazure modules due to Joseph, Polo and Mathieu. 
We apply the presentation to obtain upper bounds on
$c^\nu_{\lambda\mu}(w)$ and to study Schur positivity.
\end{abstract}

\maketitle

\section{Introduction}

Let $\lie g$ be a symmetrizable Kac--Moody Lie algebra over $\mathbb{C}$. 
For a dominant integral weight $\lambda$, let $V(\lambda)$ denote the
irreducible highest weight $\lie g$-module of highest weight $\lambda$. 
Consider the tensor product $V(\lambda) \otimes V(\mu)$ of two irreducible modules. Kostant--Kumar modules are submodules of this tensor product, indexed by elements $w$ of the Weyl group of $\lie g$:
\[
    K(\lambda,w,\mu)\;:=\;\lie{Ug}\,(v_\lambda\otimes v_{w\mu})
    \;\subseteq\; V(\lambda)\otimes V(\mu),
\]
where $v_\lambda$ and $v_{w\mu}$ are nonzero vectors of weights $\lambda$
and $w\mu$ in $V(\lambda)$ and $V(\mu)$ respectively. These first arose in the context of the Parthasarathy-Ranga Rao-Varadarajan (PRV) conjecture and its strengthening due to Kostant which asserted that the irreducible representation $V(\overline{\lambda+w\mu})$ occurs with multiplicity 1 in $K(\lambda,w,\mu)$; here $\overline{\lambda+w\mu}$ denotes the unique dominant integral weight in the Weyl group orbit of $\lambda+w\mu$. We refer the reader to \cite{apoorva-prv} for a lucid survey.

Kostant--Kumar modules form a natural family interpolating between the
Cartan component and the full tensor product. At the two extremes one has
\[
    K(\lambda,1,\mu)=V(\lambda+\mu)
    \qquad\text{and}\qquad
    K(\lambda,w_0,\mu)=V(\lambda)\otimes V(\mu),
\]
the latter in finite type, with $w_0$ being the longest Weyl group element. 
The submodules $\{K(\lambda,w,\mu)\}_{w \in W}$ form an increasing filtration of
$V(\lambda)\otimes V(\mu)$ indexed by the Bruhat poset on the double coset space
$W_\lambda\backslash W/W_\mu$, where $W_\lambda, W_\mu$ denote the stabilizers of $\lambda, \mu$ respectively. Writing $\lrw$ (resp. $\lr$) for the multiplicity of
$V(\nu)$ in $K(\lambda,w,\mu)$ (resp. $V(\lambda) \otimes V(\mu)$), one obtains a family of \emph{$w$-refined}
tensor product multiplicities. The $\lrw$ are monotone along the
Bruhat order and recover the classical $\lr$ for all $w$ of large enough length.

The refined coefficients $\lrw$ have so far been studied chiefly through
combinatorics. Formulas in terms of Littelmann's path model and
Kashiwara's crystals go back to Littelmann \cite{L1,L2} and
Joseph \cite{J2}, and were investigated in detail in \cite{KRV}. Our
aim in this paper is to develop the complementary \emph{module-theoretic} approach to
 Kostant--Kumar modules: to describe their multiplicity spaces
intrinsically, to present them by generators and relations, and to study
the natural maps between them, with applications to Schur positivity.
While many of our results are valid for all symmetrizable
Kac--Moody algebras, the presentation theorem and its consequences require
$\lie g$ to be finite-dimensional semisimple or a
Kac--Moody algebra with symmetric generalized Cartan matrix; this restriction is inherited from the
"existence of excellent-filtrations" results on which the arguments rest, and is expected
to be removable.

We now describe the main results in more detail, referring the reader to the main text of the paper for unexplained notation.

\medskip
\noindent\textbf{I. Multiplicity spaces.}
In finite type, a classical construction due to Parthasarathy-Ranga Rao-Varadarajan and Kostant realizes tensor-product multiplicities inside a single
irreducible: the space
\[
    \mcd \;=\; \bigl\{\,v\in V(\mu)_{\nu-\lambda}\;:\;
    e_i^{\,\langle\lambda,\alpha_i^\vee\rangle+1}\,v=0\ \text{ for all simple roots } \alpha_i\,\bigr\} \;\subseteq V(\mu)
\]
has dimension $\lr$. This was extended to all symmetrizable Kac--Moody algebras
by Jeralds and Kumar \cite{jeralds-kumar}.

Our first main result identifies an analogous space computing the refined
multiplicity $\lrw$. Let $V_w(\mu)=\lie{Ub}\,v_{w\mu}\subseteq V(\mu)$ be the
Demazure module, and let $\pi_w\colon V(\mu)\twoheadrightarrow V_w(\mu)$ be
the orthogonal projection with respect to the Shapovalov (contravariant)
form. Then
$
    \mcdw \;:=\; \pi_w(\mcd)
    \qquad\text{satisfies}\qquad
    \dim\mcdw=\lrw
$ 
(Theorem~\ref{thm:shap-proj-w}). It is worth emphasizing that $\mcdw$ is
realized as a \emph{quotient} of $\mcd$, not as a subspace: the naive
guess $\mcdw\stackrel{?}{=}\mcd\cap V_w(\mu)$, suggested by the crystal-side
picture, is already false for $\mathfrak{sl}_3$. 

\medskip
\noindent\textbf{II. A presentation by generators and relations.}
Our second main result is an explicit presentation of Kostant--Kumar
modules. Its natural predecessor is the presentation of Demazure modules
due to Joseph \cite{J1}, Polo \cite{P1} and Mathieu \cite{M1},
which realizes $V_w(\mu)$ as $\lie{Ub}/I$ for an explicit left ideal $I$ of
$\lie{Ub}$ generated by suitable elements of the Cartan subalgebra and powers of raising operators. Passing from the Borel to the full algebra --- and from the Demazure module to the
Kostant--Kumar module --- amounts to adjoining lowering relations. We prove
(Theorem~\ref{thm:gen-rel}) that, for $\lie g$ finite-dimensional semisimple
or symmetric Kac--Moody,
\[
    K(\lambda,w,\mu)\;\cong\;\lie{Ug}/J,
\]
where $J=\ann_{\lie{Ug}}(v_\lambda\otimes v_{w\mu})$ is generated by the
explicit raising relations $e_\alpha^{\,p_\alpha(w\mu)+1}$
($\alpha\in\posreal$) and ${\lie g}_\beta$ ($\beta\in\posimag$), the lowering relations
$f_\alpha^{\,\langle\lambda,\alpha^\vee\rangle+q_\alpha(w\mu)+1}$
($\alpha\in\posreal$) and the
Cartan relations $h-\langle\lambda+w\mu,h\rangle\,1$. At the extremes $w=1$
and $w=w_0$ (finite type) these recover, respectively, the standard presentation of the
Cartan component $V(\lambda+\mu)$ and the classical
  relations for the tensor product \cite{GE}, \cite{PRV}.

The key tool, of independent interest, is the isomorphism
\[
    \res\colon\ \Hom_{\lie{Ug}}\bigl(K(\lambda,w,\mu),V(\nu)\bigr)
    \;\xrightarrow{\ \sim\ }\;
    \Hom_{\lie{Ub}}\bigl(v_\lambda\otimes V_w(\mu),V(\nu)\bigr)
\]
(Theorem~\ref{thm:UbUg}). For
finite-dimensional $\lie g$, this is a theorem of Kumar \cite{Kum1}; we extend it to
symmetric Kac--Moody algebras using the Polo-Mathieu-Joseph excellent filtrations. In the 
symmetric finite-dimensional case (i.e., the ADE types) this gives an
algebraic alternative to Kumar's geometric argument.

\medskip
\noindent\textbf{III. Tensor envelopes, canonical morphisms, and Schur
positivity.}
Retaining only the simple-root generators of $J$ yields a smaller ideal
$J_s\subseteq J$ and a module $T(\lambda,w,\mu):=\lie{Ug}/J_s$, the
\emph{tensor envelope} of $K(\lambda,w,\mu)$. In finite type it is again a
tensor product of irreducibles,
\[
    T(\lambda,w,\mu)\;\cong\;V\!\bigl(\lambda+(w\mu)^+\bigr)\otimes V\!\bigl((w\mu)^-\bigr)^{*},
\]
and surjects canonically onto $K(\lambda,w,\mu)$
(Proposition~\ref{prop:tensenv}). This yields a $w$-dependent upper bound
on refined multiplicities by ordinary tensor-product multiplicities which is often
sharper than the naive bound $\lrw\le\lr$.

The maps between Kostant--Kumar modules that send cyclic generator to
cyclic generator --- which we call \emph{canonical morphisms} -- are governed by a
simple criterion in terms of the presenting ideals
(Proposition~\ref{prop:canmor}), from which we recover and extend a
surjection of Mathieu \cite{M2}. Finally we turn to Schur positivity in the
sense of the existence of a $\lie{Ug}$-surjection
$K(\lambda_1,w_1,\mu_1)\twoheadrightarrow K(\lambda_2,w_2,\mu_2)$. We give a
necessary condition (Proposition~\ref{prop:schurpos-kk-nec}) and a
sufficient condition of ``$\sigma$-twisted'' type (Proposition~\ref{prop:sigma}). These
Weyl-twisted maps do not, however, account for all known scenarios in which Schur positivity is known: for
instance in type $A$, the ``sort'' inequalities of Lam, Postnikov and
Pylyavskyy \cite{LPP} produce Schur-positive differences with no
corresponding surjection of twisted type.

\medskip
\noindent\textbf{Organization of the paper:}
Section~\ref{sec:prelims} fixes notation and recalls the Shapovalov form and its
basic properties.
The multiplicity results of Part~I are proved in
Section~\ref{sec:multlrw}. Part~II occupies Section~\ref{sec:genrelkk}: the
presentation theorem is stated in Section~\ref{sec:mainthmgenrel} and proved through a
sequence of reductions culminating in Section~\ref{sec:proofends}. Tensor envelopes,
canonical morphisms and the Schur-positivity results of Part~III are
developed in Sections~\ref{sec:tensenvelope}--\ref{sec:canmors}.

\medskip
\noindent\textbf{Tool and computational resource disclosure:} LLMs were used during the writing process to polish sentences, find typos and check references. The mathematics in this paper was entirely human-produced, with no involvement of AI tools at any stage. Sagemath \cite{sagemath} was employed to compute examples and build intuition for Theorem~\ref{sec:mainthmshap}.

\section{Preliminaries}\label{sec:prelims}
Unless stated otherwise, the notations of this section will be used throughout.

\subsection{} Let $\mathfrak g$ be a symmetrizable Kac--Moody Lie algebra, with triangular decomposition $\mathfrak{g}=\mathfrak{n}^{-}\oplus \mathfrak{h} \oplus\mathfrak{n}^{+}$. Let $\mathfrak{b}=\mathfrak h\oplus \mathfrak{n}^+$ be its Borel subalgebra. Let $\alpha_i, \,i=1,2,\ldots,n$ be the simple roots of $\mathfrak{g}$. Let $\Delta^+$ and $\Delta^-$ be the sets of positive and negative roots respectively and $\posreal, \,\posimag$ denote the sets of positive real and imaginary roots. For each root $\beta$, let ${\lie g}_\beta$ denote the corresponding root space of ${\lie g}$. Given $\alpha \in \posreal$, we let $e_\alpha, f_\alpha$ denote nonzero elements of ${\lie g}_\alpha$ and ${\lie g}_{-\alpha}$ respectively.

Let $Q^+ = \mathbb{Z}_{\geq 0} (\Delta^+)$ denote the positive root lattice. Let $W$ denote the Weyl group of $\lie g$, generated by the simple reflections $s_i$ for $i=1, \ldots, n$. The set of weights is denoted by $P$ and the set of dominant weights by $P^+$. The set $X = WP^+$ comprising $W$-conjugates of dominant weights is the Tits cone of $\lie g$ (or more precisely the set of integral weights in the Tits cone). Given a weight $\gamma$ in $X$, we let $\overline{\gamma}$ denote the unique dominant element in $W\gamma$ (its $W$-orbit). For a dominant weight $\lambda,$ let $V(\lambda)$ be the irreducible integrable module of highest weight $\lambda$.

Let $\geq$ denote the usual partial order on $P$, defined via $\alpha \geq \beta$ if $\alpha - \beta \in Q^+$. We will also let $\sigma \geq \tau$ to denote the Bruhat order on $W$ for $\sigma, \tau \in W$. When $\lie g$ is finite dimensional, we let $w_0$ denote the longest element of $W$.

The Demazure module $V_w(\mu)$ is defined to be the $\mathfrak{Ub}$-submodule of $V(\mu)$ generated by a vector $v_{w\mu}$ of weight $w\mu$. If the generalized Cartan matrix (GCM) of $\lie g$ is symmetric, we will refer to $\lie g$ as a {\em symmetric} Kac--Moody algebra.

\subsection{}
Let $\lie g$ be a symmetrizable Kac--Moody algebra and let $\lambda, \mu \in P^+$. Kostant--Kumar modules are defined as $\mathfrak{Ug}$-submodules of the tensor product $V=V(\lambda)\otimes V(\mu)$, generated by a single vector of the form $v_\lambda\otimes v_{w\mu},$ where $w\in W,$ the Weyl group of the Lie algebra, and $v_\lambda$ is a non-zero vector of weight $\lambda$ in $V(\lambda)$ and $v_{w\mu}$ is a non-zero vector of weight $w\mu$ in $V(\mu).$ They are denoted by \[K(\lambda,w,\mu)=\mathfrak{Ug}(v_{\lambda}\otimes v_{w\mu}).\]

For $\lambda \in P^+$, let $W_\lambda$ denote the stabilizer of $\lambda$ in $W$. We recall the following properties from \cite{KRV}, \cite{KRV2}:
\begin{proposition}\label{prop:kkprops}
Let $w, \sigma, \tau \in W$ and  $\lambda, \mu \in P^+$.
\begin{enumerate}
\item $K(\lambda,w,\mu) = \mathfrak{Ug}(v_{\lambda}\otimes V_w(\mu))$.
\item $K(\lambda, w, \mu) = \lie{Ug}(v_{\sigma \lambda} \otimes v_{\sigma w\mu})$ for all $\sigma \in W$.
\item  $K(\lambda,\sigma,\mu) = K(\lambda,\tau,\mu)$ if $W_\lambda \sigma W_\mu = W_\lambda \tau W_\mu$.
\item $K(\lambda,\sigma,\mu) \subseteq K(\lambda,\tau,\mu)$  if $W_\lambda \sigma W_\mu \leq W_\lambda \tau W_\mu$ in the Bruhat poset  $W_\lambda \backslash W /W_\mu$ of double cosets.
\item $K(\lambda,1,\mu) \cong V(\lambda+\mu)$ and when $\lie g$ is finite-dimensional, $K(\lambda,w_0,\mu) = V(\lambda) \otimes V(\mu) $.
\end{enumerate}
\end{proposition}

The following is Kostant's strengthening of the PRV conjecture and was established independently by Kumar \cite{Kum1} and Mathieu \cite{M1}.
\begin{theorem}\label{thm:kprv} (Kumar, Mathieu) 
Let $\lie g$ be a symmetrizable Kac--Moody algebra and $\lambda, \mu \in P^+$, $w \in W$. Then $V(\overline{\lambda+w\mu})$ occurs with multiplicity 1 in the irreducible decomposition of $K(\lambda,w,\mu)$. 
\end{theorem}

\subsection{}
Let $\omega_0$ denote the anti-linear automorphism of $\mathfrak g$ defined by mapping the Chevalley generators as follows $e_i\mapsto f_i, \,f_i\mapsto e_i, \;i=1,2,\ldots,n$ and $\omega_0(h)=h, \;h\in \mathfrak{h}_{\mathbb{R}}$ where ${\lie h}_\mathbb{R}$ is the real form of the Cartan subalgebra $\lie h$.

Let $\catoint$ denote the category of integrable $\lie g$-modules in category $\cato$. An integrable highest weight module $V(\lambda)$ of $\lie g$ admits a positive definite Hermitian form $\langle \cdot, \cdot \rangle$ (the {\em Shapovalov form}) satisfying the $\omega_0$-contravariance property:
 \[ \langle Xv, v' \rangle = \langle v, \omega_0(X) v' \rangle \]
    for all $v, v' \in V$ and $X \in \lie g$.
This form on $V(\lambda)$ is unique up to scaling. A module $V \in \catoint$ is a direct sum of highest weight irreducible modules:
\[ V \cong \bigoplus_{\lambda \in P^+} V(\lambda)^{\oplus k_\lambda}.\]
We let $[V:V(\lambda)] :=k_\lambda$ denote the multiplicity of $V(\lambda)$ in this decomposition.

A positive definite Hermitian $\omega_0$-contravariant form can be defined on $V$ by defining it on its components, though it is no longer unique up to scaling. We will call any such form a Shapovalov form on $V$.
See Sections 2.7, 11.5 of \cite{Kac} and Section 3.14 of \cite{Hum2} for more details.

We record below the key properties of this form that will be used in the sequel:
\begin{lemma}\label{lem:shap-props}
    Let $\lie g$  be a symmetrizable Kac--Moody algebra, $V \in \catoint$ and fix  a Shapovalov form $\langle \cdot, \cdot \rangle$ on $V$. Let $U$ be a submodule of $V$. Then:
    \begin{enumerate}
        \item Distinct weight spaces of $V$ are mutually orthogonal.
        \item $U^\perp$ is also a submodule of $V$ and 
        the orthogonal projection $\pi: V \twoheadrightarrow U$ is $\lie{Ug}$-linear. 
        \item For $\gamma \in P^+$, let $U^+_\gamma$ and $V^+_\gamma$ denote the spaces of highest weight vectors of weight $\gamma$ in $U, \, V$ respectively. Then $\pi(V^+_\gamma) = U^+_\gamma$.
    \end{enumerate}
\end{lemma}

\section{Multiplicities in tensor products and Kostant--Kumar modules}\label{sec:multlrw}

\subsection{}
When $\lie g$ is a finite-dimensional semisimple Lie algebra, the following proposition is a result independently due to Kostant and Parthasarathy-Ranga Rao-Varadarajan (see \cite[\S 2.2]{PRV}, \cite[Remark 4]{kostant-tensmult}, \cite{Kum2}). We shall therefore refer to the space $\mcd$ below as the KPRV multiplicity space. For arbitrary symmetrizable Kac-Moody algebras $\lie g$, this was established by Jeralds and Kumar \cite[Proposition 4.1]{jeralds-kumar}.

\begin{proposition}\label{prop:kashi-glob}
    Let $\mathfrak{g}$ be a symmetrizable Kac--Moody algebra. For $\lambda, \mu, \nu \in P^+$,  let
    $$\mcd=\{v\in V(\mu)_{\nu-\lambda}:\, e_i^{\langle \lambda, \,\alpha_i^{\vee}\rangle +1}\,v=0 \;\text{ for } i=1,2, \ldots, n\}.$$ Then, $\dim \mcd =c_{\lambda\mu}^{\nu},$ the multiplicity of $V(\nu)$ in $V(\lambda)\otimes V(\mu).$
\end{proposition}

Our objective is to establish the corresponding result for the $w$-refined tensor product multiplicities $\lrw$. This is accomplished in Theorem~\ref{thm:shap-proj-w} below.

\subsection{}

We state a generalization of a result of Gould-Edwards \cite[Lemma 1]{GE} which plays a key role in the sequel. 
\begin{lemma}\label{lem:gould-injectivity}
Let $\mathfrak{g}$ be a symmetrizable Kac--Moody algebra and $V \in \catoint$. Let $Z \subseteq V$ be a $\lie{Ub}$-submodule such that $Z$ generates $V$ as a $\lie{Ug}$-module. Let $\pi: V \twoheadrightarrow Z$ be the orthogonal projection with respect to a Shapovalov form on $V$. Then $\pi$ is injective on $V^+:=\{v \in V: e_i v=0\; \text{ for } i=1, \ldots, n\}$. 
\end{lemma}
\begin{proof}
We need to establish that $V^+ \cap Z^\perp = 0$. Let $v \in V^+ \cap Z^\perp$. Since $\lie{Ug}\,Z = V$, we can write $v = \sum_{i=1}^k u_i z_i$ for some $u_i \in \lie{Ug}$, $z_i \in Z$. Since $Z$ is $\lie{Ub}$-invariant and $\lie{Ug} = \lie{Un}^- \otimes \lie{Ub}$, we can assume without loss of generality that $u_i \in \lie{Un}^-$. We compute:
\[ \langle v, v \rangle = \sum_{i=1}^k \langle v, \,u_i z_i \rangle = \sum_{i=1}^k \langle \omega_0(u_i) \,v, \,z_i \rangle.\]
Now $\omega_0(u_i) \,v \in \mathbb{C}v$ since $\omega_0(u_i) \in \lie{Un}^+$ and $v \in V^+$. Since $v \in Z^\perp$, we conclude $\langle \omega_0(u_i)\, v, \,z_i \rangle =0$ for all $i$. Thus $\langle v, v \rangle =0$.  
\end{proof}

\begin{corollary}\label{cor:hwtVZ}
    With notation as in Lemma~\ref{lem:gould-injectivity}, let $\gamma\in P^+$ such that $[V:V(\gamma)] >0$. Then $\gamma$ is a weight of $Z$. 
\end{corollary}
\begin{proof}
    The hypothesis implies that there exists a nonzero $v \in V^+$ of weight $\gamma$. The image $\pi(v)$ under the projection is thus a nonzero vector of weight $\gamma$ in $Z$.
    \end{proof}
\subsection{}
In general, there are many choices for Shapovalov forms on a module $V \in \catoint$. If $V = V(\lambda_1) \otimes V(\lambda_2)$ for some $\lambda_1, \lambda_2 \in P^+$, then we will always take the following "canonical" choice - the product of the Shapovalov forms on $V(\lambda_1)$ and $V(\lambda_2)$:
\[\langle v_1\otimes v_2, v_1'\otimes v_2'\rangle = \langle v_1,v_1' \rangle \langle v_2,v_2' \rangle\]
for $v_i, v_i' \in V(\lambda_i)$ for $i=1, 2$.

We will also need the following lemma, which is a simple fact from the representation theory of $\mathfrak{sl}_2$.
\begin{lemma}\label{lem:sl2string}
Let $\mathfrak{g}$ be a symmetrizable Kac--Moody algebra, $V \in \catoint$ be a $\mathfrak{g}$-module and $\alpha \in \posreal$. Let $v \in V$ be a weight vector of weight $\gamma$ and let $k \geq 0$. Then $e_\alpha^k(v) =0$ if and only if $f_\alpha^{\langle \gamma, \alpha^\vee \rangle + k}\, v =0$.
\end{lemma}
\begin{proof}
We consider the copy of $\mathfrak{sl}_2$ spanned by $e_\alpha, f_\alpha, \alpha^\vee$.
Since $V \in \catoint$, it decomposes into a direct sum of irreducible $\mathfrak{sl}_2$-modules. In particular, the direct sum of the weight spaces of $V$ with weights in the "$\alpha$-string" $\gamma + \mathbb{Z}\alpha$ is $\mathfrak{sl}_2$-invariant. The lemma now follows from standard facts about irreducible $\mathfrak{sl}_2$-modules \cite{Hum1}.
\end{proof}

The following result follows from Proposition~\ref{prop:kashi-glob} and Lemma~\ref{lem:gould-injectivity}. The arguments closely follow that of \cite[Theorem 2]{GE}.
\begin{proposition}\label{prop:pi0-image}
    Let $\mathfrak{g}$ be a symmetrizable Kac--Moody algebra and $\lambda, \mu, \nu \in P^+$. 
    Let $v_\lambda$ denote a nonzero highest weight vector of $V(\lambda)$. Let $\pi_0: V(\lambda)\otimes V(\mu)\twoheadrightarrow v_\lambda\otimes V(\mu)$ be the orthogonal projection with respect to the product Shapovalov form. Let $\left(V(\lambda)\otimes V(\mu)\right)_{\nu}^+$  be the set of highest weight vectors of weight $\nu$ in the tensor product. Then, $\pi_0(\left(V(\lambda)\otimes V(\mu)\right)_\nu^+)= v_{\lambda}\otimes \mcd.$
\end{proposition}
\begin{proof}
Let $Z = v_\lambda \otimes V(\mu)$ and $V = V(\lambda) \otimes V(\mu)$. It is easy to see that $Z$ is $\lie{Ub}$-invariant and that $\lie{Ug} \,Z = V$. Lemma~\ref{lem:gould-injectivity} implies that $\pi_0$ is injective on $V_\nu^+$. We claim that $\pi_0(V^+_\nu) \subset v_\lambda \otimes \mcd$. To prove the claim, we observe that $Z$ is $\lie{Ub}$-invariant, and hence that $Z^\perp$ is $\lie{Un}^-$-invariant by the $\omega_0$-contravariance of the Shapovalov form. Thus $\pi_0 N (1-\pi_0)=0$ for all $N \in \lie{Un}^-$. Fix a simple root $\alpha_i$ and let  $N = f_i^{\langle \nu, \alpha_i^\vee\rangle +1} \in \lie{Un}^-$. If $\xi \in V^+_\nu$, then $N(\xi) =0$ . Letting $\pi_0(\xi) = v_\lambda \otimes u$, we compute 
\[ 0 = \pi_0 N (\xi) = \pi_0 N \pi_0(\xi) = v_\lambda \otimes N(u). \]
Thus $N(u) = 0$, i.e., $f_i^{\langle \nu, \alpha_i^\vee\rangle +1} u =0$. Since the weight of $u$ is $\nu-\lambda$, Lemma~\ref{lem:sl2string} implies that $e_i^{\langle \lambda, \alpha_i^\vee\rangle +1} u =0$. Since this holds for all simple roots, this establishes our claim that $\pi_0(V^+_\nu) \subset v_\lambda \otimes \mcd$.

Now 
$\dim V_\nu^+$ is the multiplicity of $V(\nu)$ in $V$, i.e., $\dim V_\nu^+ = c_{\lambda\mu}^\nu$, and this in turn coincides with $\dim \mcd$ by Proposition~\ref{prop:kashi-glob}. Hence $\pi_0$ maps $V^+_\nu$ isomorphically to $v_\lambda \otimes \mcd$. 
\end{proof}

\subsection{}
 Let $\mathfrak{g}$ be a symmetrizable Kac--Moody algebra and let $\zeta \in P^+$. Consider the irreducible representation $V(\zeta)$. Given $\sigma \in W$, let $v_{\sigma \zeta}$ denote the unique (upto scaling) vector of weight $\sigma\zeta$ in $V(\zeta)$. 
Now, suppose $\lambda, \mu \in P^+$, $w \in W$. Consider the Kostant--Kumar module \[K(\lambda, w, \mu)= \mathfrak{Ug}(v_\lambda\otimes v_{w\mu}) \subseteq V(\lambda) \otimes V(\mu).\] 
We recall from Proposition~\ref{prop:kkprops} that $K(\lambda,w,\mu)$ may equivalently be defined as $\mathfrak{Ug}(v_{\sigma\lambda} \otimes v_{\sigma w\mu})$ for each fixed $\sigma \in W$. 
Given $\nu \in P^+$, let $\lrw :=[K(\lambda,w,\mu): V(\nu)]$ be the multiplicity of $V(\nu)$ in $K(\lambda, w, \mu)$:
 \begin{equation}\label{eq:kkdec}
      K(\lambda,w,\mu) = \bigoplus_{\nu \in P^+} V(\nu)^{\oplus \lrw}.
 \end{equation}
Given a weight $\gamma$ in the Tits cone, we recall that $\overline{\gamma}$ denotes the unique dominant element in $W\gamma$ (its $W$-orbit). We now record the following simple fact:
\begin{lemma}\label{lem:supplrw}
With notation as above, if $\lrw >0$, then $\lambda+\mu \geq \nu \geq \overline{\lambda+w\mu}$. Further, when $\nu = \lambda+\mu$ or $\nu = \overline{\lambda+w\mu}$, we have $\lrw = 1$.
\end{lemma}
\begin{proof}
Let $\sigma \in W$ and define $Z:= \lie{Ub}(v_{\sigma\lambda} \otimes v_{\sigma w\mu}) \subset K(\lambda,w,\mu)=:K$. Note that $Z$ generates $K$ as a $\lie{Ug}$-module  and the weights of $Z$ are all $\geq \sigma(\lambda + w\mu)$. Applying Corollary~\ref{cor:hwtVZ}, we conclude that $\nu \geq \sigma(\lambda + w\mu)$ for all $\sigma \in W$. This implies that $\nu \geq \overline{\lambda+w\mu}$. The other inequality is straightforward. The "extreme" cases follow from Theorem~\ref{thm:kprv}.
\end{proof}
In particular, Lemma~\ref{lem:supplrw} implies that even when $\lie g$ is infinite-dimensional, there are only finitely many direct summands in Equation~\eqref{eq:kkdec}.
We now seek to obtain a description of $\lrw$ that parallels Proposition~\ref{prop:kashi-glob} for $\lr$, i.e., to construct a natural space $\mcdw$ such that \[\lrw = \dim \mcdw.\]

We note that the corresponding description on the combinatorial side is well-known. 
Let $\crys(\mu)$ denote the crystal corresponding to $V(\mu)$, with highest weight element $\hwcrys[\mu] \in \crys(\mu)$. We let $\tilde{e}_i, \tilde{f}_i$ denote the crystal raising and lowering operators corresponding to the simple root $\alpha_i$. Given $w \in W$ with a reduced word $w = s_{i_1} s_{i_2} \cdots s_{i_k}$, let 
\[ \crys_w(\mu) = \left\{ \tilde{f}_{i_1}^{n_1} \tilde{f}_{i_2}^{n_2} \cdots \tilde{f}_{i_k}^{n_k} \hwcrys[\mu]: n_i \geq 0 \text{ for all } i\right\} \subset \crys(\mu)\]
denote the Demazure crystal corresponding to $w$. We then have \cite{J2,KRV}:
\[\lr = \# \crys_{\lambda\mu}^\nu \text{ and } \lrw = \# \left( \crys_{\lambda\mu}^\nu \cap \crys_w(\mu) \right), \text{ where }\]
\begin{equation}\label{eq:crysmult}
\crys_{\lambda\mu}^\nu=\{\hwcrys \in \crys(\mu)_{\nu-\lambda}:\, \tilde{e}_i^{\langle \lambda, \,\alpha_i^{\vee}\rangle +1}\,\hwcrys=0 \;\text{ for } i=1,2, \ldots, n\}.
\end{equation}
Note that  $\dim V_w(\mu) = \# \crys_w(\mu)$, and by Proposition~\ref{prop:kashi-glob} we have $\dim \mcd = \#\crys_{\lambda\mu}^\nu$. 
This close analogy prompts one to consider the space $\mcd\cap V_w(\mu)$ as a natural 
candidate for $\mcdw$. However, it turns out that the dimension of this space does not give the correct multiplicity $\lrw$.
\begin{example}\label{Dintproj}
Let $\mathfrak{g}=\mathfrak{sl}_3$ with fundamental weights denoted $\Lambda_1, \Lambda_2$. Let $\lambda=\Lambda_1$, $ \mu=\Lambda_1+\Lambda_2$, $ w=s_1s_2,$ and $ \nu=\Lambda_1.$ Here 
$\dim (\mcd\cap V_w(\mu)) =0$ while $\lrw=1$. 
\end{example}

\subsection{}\label{sec:mainthmshap}
It turns out that $\mcdw$ is more naturally realized as a quotient rather than as a subspace of $\mcd$. The following is one of the main results of this paper. 
\begin{theorem}\label{thm:shap-proj-w}
    Let $\mathfrak{g}$ be a symmetrizable Kac--Moody Lie algebra. Let $\lambda, \mu, \nu \in P^+$ and $w \in W$.
    Let $\pi_w: V(\mu)\twoheadrightarrow V_w(\mu)$ be the orthogonal projection onto the Demazure module $V_w(\mu)$ with respect to the Shapovalov form. Then, $\mcdw:= \pi_w\left(\mcd\right)$ satisfies $\dim \mcdw=c_{\lambda\mu}^{\nu}(w)$.
    \end{theorem}

\begin{proof}
Consider the following diagram, where all indicated maps are the orthogonal projections with respect to the tensor Shapovalov form on $V(\lambda) \otimes V(\mu)$. Since each $\pi_i$ is a map from a space to its subspace, it is clear that the diagram commutes. We note that $\pi_1 = 1 \otimes \pi_w$.

\[ 
\begin{tikzcd}
V(\lambda) \otimes V(\mu) \arrow[r, two heads, "\pi_0"] \arrow[d, two heads, "\pi_2"] & v_\lambda \otimes V(\mu) \arrow[d, two heads, "\pi_1"] \\
K(\lambda,w,\mu) \arrow[r, two heads, "\pi_3"] & v_\lambda \otimes V_w(\mu)
\end{tikzcd}
\]
For convenience, we set $V = V(\lambda) \otimes V(\mu)$ and $K = K(\lambda,w,\mu)$. For $\nu \in P^+$, let $V^+_\nu$ and $K^+_\nu$ denote the spaces of highest weight vectors of weight $\nu$ in $V$ and $K$ respectively. 

By Proposition~\ref{prop:pi0-image}, $\pi_0(V_\nu^+)= v_{\lambda}\otimes \mcd$. Since $\mcdw = \pi_w(\mcd)$ by definition, we have 
\[\pi_1\pi_0(V^+_\nu) = \pi_1 (v_\lambda \otimes \mcd) = v_\lambda \otimes \mcdw.\] 
It follows from Lemma~\ref{lem:shap-props} that $\pi_2(V^+_\nu) = K^+_\nu$; since $\pi_3\pi_2 (V^+_\nu)= \pi_1\pi_0(V^+_\nu)$, we conclude that $\pi_3(K^+_\nu) = v_\lambda \otimes \mcdw$. But $\pi_3$ is injective on $K^+_\nu$ by Lemma~\ref{lem:gould-injectivity} since $v_\lambda \otimes V_w(\mu)$ is a $\lie{Ub}$-submodule of $K(\lambda,w,\mu)$ that generates it over $\lie{Ug}$.  Thus $\pi_3$ maps $K^+_\nu$ isomorphically onto $v_\lambda \otimes \mcdw$. Since $\lrw = \dim K^+_\nu$, this completes the proof of Theorem~\ref{thm:shap-proj-w}.
\end{proof}

\subsection{}
\begin{example}
Consider $\mathfrak{g}=\mathfrak{sl}_3$ with simple roots $\alpha_1, \alpha_2$. Let $\lambda = \theta, \mu = 2\theta$ where $\theta=\alpha_1+\alpha_2$ is the highest root and let $w=s_1s_2$ and $\nu=\theta$. Then, $c_{\theta,2\theta}^{\theta}(s_1s_2)=0$ as can be verified from \eqref{eq:crysmult}. Further, $\mcd=\ker (e_1^2)\cap \ker (e_2^2)$ in the set of $0$-weight vectors of $V(2\theta)$. In this case, we can compute the Shapovalov form and obtain that 
$ \mcd\perp V_{s_1s_2}(2\theta)$, as Theorem~\ref{thm:shap-proj-w} asserts.
\end{example}

\subsection{}
We have the following corollary that provides vanishing and stability criteria for the $\lrw$. This follows readily from Theorem~\ref{thm:shap-proj-w}.
\begin{corollary}
\begin{enumerate}
    \item 
\begin{equation}\label{eq:comult}
    c_{\lambda\mu}^\nu(w)= \dim \left(\frac{\mcd}{\mcd \cap V_w(\mu)^\perp}\right).
    \end{equation}
\item $\lrw =0 \iff \mcd \perp V_w(\mu).$
\item $\lrw = \lr \iff \mcd \cap V_w(\mu)^\perp = (0)$. 
\end{enumerate}
\end{corollary}
Note that the second assertion above gives an elegant criterion for the vanishing of $\lrw$. This happens precisely when the KPRV multiplicity space $\mcd$ and the Demazure module $V_w(\mu)$ are orthogonal relative to the Shapovalov form.

It turns out that the space in Equation~\eqref{eq:comult} has properties that are better aligned to those of the Kostant--Kumar module $K(\lambda,w,\mu)$.  We have for instance ({\em cf.} Proposition~\ref{prop:kkprops}):
\begin{proposition}
    The space $\frac{\mcd}{\mcd \cap V_w(\mu)^\perp}$ is invariant with respect to the choice of representative $w$ in the double coset $W_\lambda wW_\mu$ where $W_\lambda$ and $W_\mu$ are stabilizers  in $W$ of $\lambda$ and $\mu$ respectively. 
\end{proposition}
\begin{proof}
Let $\mcmw := \mcd \cap V_w(\mu)^\perp$. Since (i) $V_{w\tau}(\mu) = V_w(\mu)$ for all $\tau \in W_\mu$ and (ii) $W_\lambda$ is generated by the simple reflections it contains, it suffices to prove that $\mcmw[s_iw] = \mcmw$ for all $s_i \in W_\lambda$. Replacing $w$ by $s_i w$ if necessary, we may assume without loss of generality that $s_iw>w$ in the Bruhat order on $W$. In this case, 
$V_{s_iw}(\mu)=\mathrm{span} \, \{f_i^n v: v \in V_w(\mu), n \geq 0\} \supset V_w(\mu)$. Consequently, $\mcmw[s_iw] \subset \mcmw$.

 Conversely, given $v\in \mcmw$, we need to show that $v \perp V_{s_iw}(\mu)$. Note that since $s_i \in W_\lambda$, we have $\langle \lambda, \alpha_i^\vee \rangle =0$. A typical element $y \in V_{s_iw}(\mu)$ can be written $y = f_i^kx$ for some $x \in V_w(\mu)$ and $k\geq 0$. If $k=0,$ $y \in V_w(\mu)$ and $v \perp y$.  If $k>0,$ $\langle v,f_i^kx \rangle= \langle e_iv,f_i^{k-1}x\rangle$. Now, $v\in \mcd$ ensures that $e_i^{\langle\lambda,\alpha_i^\vee\rangle+1}v= e_i v = 0$. Thus $v \perp y$ in this case as well. 
 \end{proof}

\begin{rem}
We may thus view $\frac{\mcd}{\mcd \cap V_w(\mu)^\perp}$ as a canonical "model" for $\lrw$, possessing invariance under choice of double-coset representative in $W_\lambda w W_\mu$. In contrast, the space $\pi_w(\mcd)$ of Theorem~\ref{thm:shap-proj-w} is not invariant; only its dimension remains the same. For instance, we consider the modules of example~\ref{Dintproj}, with $\lambda = \Lambda_1$ and $\mu = \Lambda_1 + \Lambda_2$ for $\lie g = \lie{sl}_3$. Since $s_2$ is in the stabilizer of $\lambda,$ $K(\lambda,s_2s_1s_2,\mu)=K(\lambda,s_1s_2,\mu).$ However, $\pi_{s_1s_2s_1}(\mcd)\neq \pi_{s_1s_2}(\mcd)$; the former is all of $\mcd$ while the latter cannot equal $\mcd$ as Example~\ref{Dintproj} shows.  
\end{rem}

\section{A presentation for Kostant--Kumar modules}\label{sec:genrelkk}
\subsection{Demazure modules}
Let  $\mathfrak{g}$ be a symmetrizable Kac--Moody algebra. Given $\gamma \in P$ and $\alpha \in \posreal$, define 
\begin{align}
    p_\alpha(\gamma) &:= \max (0, -\langle \gamma, \alpha^\vee\rangle) \label{eq:palpha}\\
    q_\alpha(\gamma) &:= \max (0, \langle \gamma, \alpha^\vee\rangle) \label{eq:qalpha}
\end{align}

The following presentation for Demazure modules was proved by Joseph \cite{J1} (characteristic zero) and Polo \cite{P1} (characteristic free) in the finite-dimensional case. Mathieu \cite[Lemma 26]{M1} showed that Polo's argument can be easily adapted to establish the corresponding result for all symmetrizable Kac--Moody Lie algebras. 

\begin{proposition}\label{prop:dempres}
    Let $\mathfrak{g}$ be a symmetrizable Kac--Moody algebra and let $\mu\in P^+, w\in W$.  Consider the Demazure module $V_w(\mu)=\mathfrak{Ub} \,v_{w\mu}$. 
    Let $I$ be the left ideal of $\mathfrak{Ub}$ generated by the elements \begin{enumerate}
        \item $e_\alpha^{\,p_\alpha (w\mu)+1}, \;\; \alpha \in \posreal$.
        \smallskip\item ${\lie g}_\beta, \;\; \beta \in \posimag$.
        \smallskip\item $ h - \langle w\mu, h \rangle 1, \;\; h \in \mathfrak{h}$.
    \end{enumerate}
    Then $I=\ann_{\,\mathfrak{Ub}}(v_{w\mu})$, and $V_w(\mu) \cong \mathfrak{Ub}/I$.
\end{proposition}

\subsection{Generators and relations for Kostant--Kumar modules}\label{sec:mainthmgenrel}
The following result, which gives a presentation for Kostant--Kumar modules is one of the main results of this paper.

\begin{theorem}\label{thm:gen-rel}
    Let $\mathfrak{g}$ be a finite dimensional semisimple Lie algebra or a symmetric Kac--Moody Lie algebra. Consider the Kostant--Kumar module $K(\lambda,w,\mu)=\mathfrak{Ug} (v_\lambda\otimes v_{w\mu})$. Let $J$ be the left ideal of $U\mathfrak g$ generated by the elements:  
    \begin{enumerate}
        \item $e_\alpha^{\,p_\alpha (w\mu)+1}, \;\; \alpha \in \posreal$.
        \smallskip\item $f_\alpha^{\,\langle \lambda, \alpha^\vee \rangle + q_\alpha (w\mu)+1}, \;\; \alpha \in \posreal$.
        \smallskip\item ${\lie g}_\beta, \;\; \beta \in \posimag$.
        \smallskip\item $ h - \langle \lambda + w\mu, h \rangle 1, \;\; h \in \mathfrak{h}$.
    \end{enumerate}
    Then $J=\ann_{\,\mathfrak{Ug}} (v_\lambda \otimes v_{w\mu})$, and $K(\lambda, w, \mu) \cong \mathfrak{Ug}/J$.
\end{theorem}

Before embarking on the proof of Theorem~\ref{thm:gen-rel}, two well-known special cases are worth mentioning. These are the subject of the following remarks.
\begin{rem}\label{rem:hwtreprel}
If $\lie g$ is as in Theorem~\ref{thm:gen-rel} (or more generally an arbitrary symmetrizable Kac--Moody algebra) and $w=1$, then $K(\lambda,w,\mu) \cong V(\lambda+\mu)$. As is well-known \cite[Corollary 10.4]{Kac}, the annihilator of $v_\lambda \otimes v_\mu$ in $\lie{Ug}$ is generated by the elements:
\[ h - \langle \lambda+\mu, h \rangle \,1\;\; \text{ and } \;\;e_{\alpha}, \; f_{\alpha}^{\,\langle \lambda +\mu, \alpha^\vee \rangle + 1} \text{ for } \alpha \text{ simple and } h \in \lie h.\]
Since $p_\alpha(\mu)=0$ and $q_\alpha(\mu) = \langle \mu, \alpha^\vee \rangle$, the above relations are a subset of the ones in Theorem~\ref{thm:gen-rel}. It is easy to show that these two sets of relations generate the same left ideal of $\lie{Ug}$.
\end{rem}

\begin{rem}\label{rem:prvrel}
When $\lie g$ is finite-dimensional and $w=w_0$, we have $K(\lambda, w_0, \mu) \cong V(\lambda) \otimes V(\mu)$. In this case, the annihilator $\ann_{\,\mathfrak{Ug}} (v_\lambda \otimes v_{w_0\mu})$ is already generated by the following subset of the elements listed in Theorem~\ref{thm:gen-rel}:
\[ h - \langle \lambda+w_0\mu, h \rangle \,1\;\; \text{ and } \;\;e_{\alpha}^{-\langle w_0\mu, \alpha^\vee \rangle + 1}, \; f_{\alpha}^{\,\langle \lambda, \alpha^\vee \rangle + 1} \text{ for } \alpha \text{ simple and } h \in \lie h.\]
In others words, in relations (1), (2) of Theorem~\ref{thm:gen-rel}, we only take  the Chevalley generators corresponding to the simple roots. 
This result is implicit in \cite{PRV} (see also \cite{Kum2, GE}), and readily follows from Theorem~\ref{thm:gen-rel}. 
\end{rem}

Finally, we remark that when $\mathfrak{g}$ is infinite-dimensional, the restriction that the GCM of $\mathfrak{g}$ be symmetric is inherited from a result on excellent filtrations in this setting (Proposition~\ref{prop:excellent-joseph} below). This latter result is expected to hold without this assumption.

The proof of Theorem~\ref{thm:gen-rel} occupies Sections \ref{sec:oint}, \ref{sec:UbUg} and \ref{sec:zzz}.

\subsection{} \label{sec:oint}
As a first step in the proof of Theorem~\ref{thm:gen-rel}, we establish:

\begin{proposition}\label{prop:catoint}
In the notation of Theorem~\ref{thm:gen-rel}, the $\mathfrak{g}$-module $\mathfrak{Ug}/J$ is in $\catoint$. 
\end{proposition}

\begin{proof}
Let $v$ denote the coset of $1$ in $M:=\mathfrak{Ug}/J$. Then $M = \lie{Ug} \,v$. 
Recall that the Cartan subalgebra $\lie{h}$ acts diagonalizably on $\lie{Ug}$ by the adjoint action. We also have that $hv = \langle \lambda+w\mu, h\rangle v$ for all $h \in \lie h$ and so $v$ is a common eigenvector for the action of $\lie h$. Together these facts imply that $\lie{h}$ acts diagonalizably on all of $M$.
It is also clear from the explicit generators of $J$ that the $e_\alpha$ and $f_\alpha$ act nilpotently on $v$ for all $\alpha \in \posreal$.
It follows from \cite[Lemma 3.4]{Kac} that the $e_\alpha$ and $f_\alpha$ act locally nilpotently on $M$ for all $\alpha \in \posreal$.

We claim that the space $N:=\lie{Ub}\,v$ is finite-dimensional. To prove this, we observe from the generators of $J$ that: (i) ${\lie g}_\beta \,v=0$ for all $\beta \in \posimag$, (ii) $e_\alpha\, v =0$ for all $\alpha \in \posreal$ such that $\langle w\mu, \alpha^\vee \rangle \geq 0$. Thus, there are only finitely many $\alpha \in \posreal$ such that $e_\alpha v \neq 0$, since any such $\alpha$ necessarily satisfies $\langle w\mu, \alpha^\vee \rangle < 0$, and hence in particular that $w^{-1} \alpha \in \posreal[-]$; recall that the set $I(w)=\{\alpha \in \posreal: w^{-1}\alpha \in \posreal[-]\}$ has cardinality equal to $\ell(w)$, the length of $w$. Following the notation of Mathieu \cite{M1}, we denote $\mathfrak{n}(w) := \oplus_{\alpha \in I(w)} \,\lie{g}_\alpha$. This is a Lie subalgebra of $\lie n^+$ and we  conclude from the above arguments that $N=\lie{Ub}\,v = \mathfrak{Un}(w)\,v$. Since each $e_\alpha$ for $\alpha \in I(w)$ acts locally nilpotently on $N$, it follows from the PBW theorem applied to $\mathfrak{Un}(w)$ that $N$ is finite-dimensional, as claimed.

Now $M = \lie{Ug} \, v = U\lie{n}^- N$. If $\gamma_i, \,i=1, 2, \ldots, k$ denote the finitely many weights of $N$, then clearly the weights of $M$ lie in $\cup_{i=1}^k D(\gamma_i)$ with $D(\gamma) = \{\zeta \in P: \zeta \leq \gamma\}$. Applying the PBW theorem to $\lie{n}^-$ also establishes that $M$ has finite-dimensional weight spaces. This completes the verification of all properties required to establish that $M \in \catoint$.
\end{proof}

\subsection{} \label{sec:UbUg}

Let $\mathfrak{g}$ be a symmetrizable Kac--Moody algebra. Let $\lambda, \mu, \nu \in P^+$ and $w \in W$. The inclusion map $i: v_\lambda \otimes V_w(\mu) \hookrightarrow K(\lambda, w, \mu)$ is $\mathfrak{Ub}$-linear. Since its image generates $K(\lambda, w, \mu)$ as a $\mathfrak{Ug}$-module, the restriction map
$$ \res:  \Hom_{\,\mathfrak{Ug}} (K(\lambda, w, \mu), V(\nu)) \to \Hom_{\,\mathfrak{Ub}} (v_\lambda \otimes V_w(\mu), V(\nu))$$
 is injective. When $\mathfrak g$ is finite-dimensional, it is a theorem of Kumar \cite{Kum1}, \cite{Kum2} that $\res$ is an isomorphism (of vector spaces). The following theorem generalizes this result to all symmetric Kac--Moody algebras. For the simply-laced (i.e., symmetric) finite-dimensional Lie algebras, our proof affords an algebraic alternative to the geometric arguments of Kumar \cite{Kum1}.

\begin{theorem}\label{thm:UbUg}
    Let $\mathfrak{g}$ be a finite-dimensional semisimple Lie algebra or a symmetric Kac--Moody algebra. Let $\lambda, \mu, \nu \in P^+$ and $w \in W$. Then 
\begin{equation} \label{eq:resmap}
     \res:  \Hom_{\,\mathfrak{Ug}} (K(\lambda, w, \mu), V(\nu)) \to \Hom_{\,\mathfrak{Ub}} (v_\lambda \otimes V_w(\mu), V(\nu))
\end{equation}
is an isomorphism of vector spaces.
\end{theorem}

\begin{corollary}
    Under the hypotheses of Theorem~\ref{thm:UbUg}, we have 
    \[\lrw = \dim \{v \in V(\nu)_{\lambda+w\mu}: e_\alpha^{p_\alpha(w\mu)+1} v =0 \text{ for } \alpha \in \posreal \text{ and } {\lie g}_{\beta} v =0 \text{ for } \beta \in \posimag\}.\]
\end{corollary}

\subsection{}\label{sec:zzz}
As mentioned previously, Theorem~\ref{thm:UbUg} was proved by Kumar for finite-dimensional semisimple $\mathfrak{g}$. We now consider the symmetric Kac--Moody algebra case. Our proof has two main ingredients. The first is the following result (Proposition~\ref{prop:excellent-joseph}) on the existence of {\em excellent filtrations} (or {\em Demazure flags}).
When $\lie g$ is finite-dimensional, this was first conjectured by Polo and proved by Polo \cite{P1} and Mathieu \cite{M3} (see also van der Kallen \cite{vdkallen-tifr}). For $\lie g$ an infinite-dimensional Kac--Moody algebra with symmetric GCM, this was established by Joseph \cite{J2}.

\begin{proposition} \label{prop:excellent-joseph} (Polo, Mathieu, Joseph)
    Let $\mathfrak{g}$ be finite-dimensional semisimple or a symmetric Kac--Moody Lie algebra. Let $\lambda, \mu \in P^+$ and $w \in W$. The $\mathfrak{Ub}$-module $v_\lambda\otimes V_w(\mu)$ admits a filtration by $\mathfrak{Ub}$-submodules, whose successive quotients are isomorphic to Demazure modules. That is,
    $$v_\lambda\otimes V_w(\mu)=F_0\supsetneq F_1\supsetneq \cdots \supsetneq F_k = 0$$ such that $$F_{i-1}/F_{i}\cong V_{\sigma_i}(\gamma_i),$$ for some $\sigma_i\in W, \gamma_i\in P^+$ for $1 \leq i \leq k$. 
    Further, these are related to the Kostant--Kumar module by
    $$K(\lambda,w,\mu) \cong \bigoplus_{i=1}^k V(\gamma_i).$$ 
    In particular, for each fixed $\nu \in P^+$, $c_{\lambda\mu}^{\nu}(w) = \#\{1 \leq i \leq k :\, \gamma_i=\nu\}.$
\end{proposition}
We also need the following lemma concerning homomorphisms of Demazure modules. We will find   it convenient to denote $V_{\infty}(\mu):=V(\mu).$
\begin{lemma}\label{lem:demhom}
 Let $\mathfrak g$ be a symmetrizable Kac--Moody Lie algebra. Let $\lambda, \mu \in P^+$ and $\sigma \in W, \tau \in W\cup\{\infty\}$. Then 
\begin{equation*}
     \dim \Hom_{\,\mathfrak{Ub}}(V_\sigma(\lambda),V_\tau(\mu))=
    \begin{cases}
      1 & \text{ if } \lambda=\mu \text{ and } (\sigma\leq \tau \text{ in } W/W_\lambda, \text{ or } \tau=\infty)\\
      0       & \text{otherwise}
    \end{cases}
  \end{equation*}
\end{lemma}
\begin{proof}
    If there is a non-zero $\mathfrak{Ub}$-homomorphism $\phi: V_\sigma(\lambda) \to V_\tau(\mu),$ then the generator $v_{\sigma\lambda}$ is mapped to a non-zero element $\phi(v_{\sigma\lambda})$ of weight $\sigma\lambda$ in $V_\tau(\mu)$. This implies that $\sigma\lambda$ is a weight of $V_\tau(\mu) \subset V(\mu).$ Therefore, so is $\lambda$ and hence $\lambda\le \mu.$ For the reverse inequality, observe that $v_\mu$ can be obtained from $\phi(v_{\sigma\lambda})$ by applying a suitable sequence $e_{i_1} e_{i_2} \cdots e_{i_k}$ of raising operators. Applying the same sequence of operators to $v_{\sigma\lambda}$ must therefore give a non-zero vector of weight $\mu$ in $V_{\sigma}(\lambda).$ Hence, $\lambda\ge \mu,$ from which we get $\lambda=\mu.$ Now, for a non-zero map from $V_\sigma(\lambda)$ to $V_{\tau}(\lambda),$ $v_{\sigma\lambda}$ must have a non-zero image and thus, occur as a weight in $V_\tau(\lambda).$ Therefore, $\sigma\le \tau$ in $W/W_\lambda$.
\end{proof}

\subsection{}
We now use Proposition~\ref{prop:excellent-joseph} and Lemma~\ref{lem:demhom}
 to prove Theorem~\ref{thm:UbUg}. Since the restriction map of \eqref{eq:resmap}
is injective, it suffices to prove that $$\dim \Hom_{\,\mathfrak{Ub}}(v_\lambda\otimes V_w(\mu), V(\nu)) \le \dim \Hom_{\,\mathfrak{Ug}}(K(\lambda, w,\mu), V(\nu)).$$ 
Note that both the dimensions are finite. We know by Proposition~\ref{prop:excellent-joseph} that there is a $\mathfrak{Ub}$-filtration $$v_\lambda\otimes V_w(\mu)=F_0\supsetneq F_1\supsetneq \cdots \supsetneq F_k = 0$$ such that $$F_{i-1}/F_{i}\cong V_{\sigma_i}(\gamma_i),$$ for some $\sigma_i\in W, \gamma_i\in P^+$ for $1 \leq i \leq k$.
Now, consider the exact sequence $$0\to F_i\to F_{i-1}\to F_{i-1}/F_i\to 0. $$ Since $\Hom$ is a contravariant left-exact functor, we have a left exact sequence $$0\to \Hom_{\,\mathfrak{Ub}}\left(\frac{F_{i-1}}{F_i},\, V(\nu)\right)\to \Hom_{\,\mathfrak{Ub}}(F_{i-1}, V(\nu)) \to \Hom_{\,\mathfrak{Ub}}(F_{i}, V(\nu))$$ Therefore, for $i=1, \ldots, k$, we have
\begin{equation} \label{eq:telescope}
    \dim \Hom_{\,\mathfrak{Ub}}(F_{i-1}, V(\nu)) - \dim \Hom_{\,\mathfrak{Ub}}(F_i, V(\nu)) \leq \dim \Hom_{\,\mathfrak{Ub}}\left(\frac{F_{i-1}}{F_i}, V(\nu)\right).
\end{equation}
    Now by Lemma~\ref{lem:demhom}, $\dim \Hom_{\,\mathfrak{Ub}}(F_{i-1}/F_i, V(\nu))= \dim \Hom_{\,\mathfrak{Ub}}(V_{\sigma_i}(\gamma_i), V(\nu)) = 1 \text{ or } 0,$ depending on whether or not $\gamma_i = \nu$.
Appealing to Proposition~\ref{prop:excellent-joseph} again, $$\sum_{i=1}^k \dim \Hom_{\,\mathfrak{Ub}}\left(\frac{F_{i-1}}{F_i}, V(\nu)\right) = \#\{1 \leq i \leq k :\, \gamma_i=\nu\} = \dim \text{Hom}_{\,\mathfrak{Ug}}(K(\lambda, w,\mu), V(\nu)).$$
Adding the inequalities of \eqref{eq:telescope} for $i=1, \ldots, k$ completes the proof of Theorem~\ref{thm:UbUg}.

\subsection{}
We extract the essential idea of the proof of Theorem~\ref{thm:UbUg} to formulate the following lemma. This applies to more general situations where excellent filtrations with nice properties may exist. This lemma is inspired by Lemma 9.6 of Kac's book \cite{Kac}.
\begin{lemma}\label{lem:UbUg-seed}
    Let $\mathfrak{g}$ be a symmetrizable Kac--Moody algebra. Let $V \in \catoint$ and let $Z$ be a (possibly infinite-dimensional) $\lie{Ub}$-submodule of $V$ such that $Z$ generates $V$ as a $\lie{Ug}$-module. For each fixed $\nu \in P^+$, suppose the following conditions hold:
    \begin{enumerate}
        \item $Z$ admits an {\em excellent} $\nu$-{\em filtration}, i.e.,  
    a finite descending chain of $\mathfrak{Ub}$-submodules:
    $$ Z = F_0\supsetneq F_1\supsetneq \cdots  \supsetneq  F_k =0$$ 
    and a subset $J \subset \{1, 2, \ldots, k\}$ such that: 
    \begin{enumerate}
        \item if $j \in J$, $F_{j-1}/F_{j}\cong V_{\sigma_j}(\gamma_j)$ for some $\sigma_j\in W, \gamma_j\in P^+$, 
        \item if $j \notin J$, $\left(F_{j-1}/F_{j}\right)_\nu =0$. 
    \end{enumerate}
    \item  $\#\{j \in J:\, \gamma_j=\nu\} \leq \dim \Hom_{\,\mathfrak{Ug}} (V, V(\nu))$.
        \end{enumerate} 
    Then the restriction map
     $\res:  \Hom_{\,\mathfrak{Ug}} (V, V(\nu)) \to \Hom_{\,\mathfrak{Ub}} (Z, V(\nu))$ is an isomorphism of vector spaces for each $\nu \in P^+$.
\end{lemma}
 We leave the proof to the interested reader - it closely follows that of Theorem~\ref{thm:UbUg}. \qed
 
\subsection{}
We use Lemma~\ref{lem:UbUg-seed} to establish the following "limiting" version of Theorem~\ref{thm:UbUg}. When $\lie g$ is finite dimensional semisimple, Proposition~\ref{prop:invlim} below coincides with Theorem~\ref{thm:UbUg} with $w=w_0$. For  infinite-dimensional $\lie g$, this may be viewed as the "limit" of Theorem~\ref{thm:UbUg} as "$w \to \infty$".

\begin{proposition}\label{prop:invlim}
    Let $\mathfrak{g}$ be a finite dimensional semisimple Lie algebra or a symmetric Kac--Moody algebra and let $\lambda, \mu, \nu \in P^+$. Then the restriction map $$\res: \Hom_{\,\mathfrak{Ug}}(V(\lambda)\otimes V(\mu), V(\nu))\cong  \Hom_{\,\mathfrak{Ub}}(v_\lambda\otimes V(\mu), V(\nu))$$ is an isomorphism of vector spaces.
\end{proposition}
\begin{proof}
    We may assume $\dim \lie g = \infty$. Since $V_\sigma(\mu) \subset V_\tau(\mu)$ if $\sigma \leq \tau$ and $V(\mu)= \cup_{\sigma \in W} V_\sigma(\mu)$, it follows that there exists $w \in W$ such that $V(\mu)_{\nu-\lambda} = (V_w(\mu))_{\nu-\lambda}$. 

    Consider $F_1 := v_\lambda \otimes V_w(\mu)$. By Proposition~\ref{prop:excellent-joseph},  it admits an excellent filtration $F_1 \supset F_2 \supset \cdots \supset F_k=0$ with $F_{i-1}/F_i \cong V_{\sigma_i}(\gamma_i)$ for $i \geq 2$, where $\sigma_i \in W$ and $\gamma_i \in P^+$. Let $Z :=v_\lambda \otimes V(\mu)$ and $V = V(\lambda) \otimes V(\mu)$; clearly $Z$ generates $V$ as a $\lie{Ug}$-module. Now $Z/F_1 \cong v_\lambda \otimes (V(\mu)/V_w(\mu))$. By our choice of $w$, $Z/F_1$ does not have $\nu$ as a weight. Defining $F_0 = Z$, we thus obtain the filtration
    \[ Z = F_0 \supset F_1 \supset \cdots \supset F_k=0\]
which satisfies all the properties of a $\nu$-excellent filtration as stipulated in Lemma~\ref{lem:UbUg-seed} with $J=\{2, 3, \cdots, k\}$. We further have from Proposition~\ref{prop:excellent-joseph} that $\#\{j \in J: \gamma_j = \nu\} = \lrw$. Since $\lrw \leq \lr = \dim \Hom_{\,\mathfrak{Ug}} (V, V(\nu))$, the final condition of Lemma~\ref{lem:UbUg-seed} is also satisfied, and this completes the argument.
\end{proof}

\subsection{}
Recall the definition of the left ideal $J$ of $\mathfrak{Ug}$ from Theorem~\ref{thm:gen-rel}. We consider the corresponding left ideal $J'$ of $\mathfrak{Ub}$ generated by:
\begin{enumerate}
    \item $e_\alpha^{\,p_\alpha (w\mu)+1}, \;\; \alpha \in \posreal$.
        \item ${\lie g}_\beta, \;\; \beta \in \posimag$.
        \item $ h - \langle \lambda + w\mu, h \rangle 1, \;\; h \in \mathfrak{h}$.
    \end{enumerate}
We now have the following relationship between the $\mathfrak{Ug}$-module $\mathfrak{Ug}/J$ and the $\mathfrak{Ub}$-module $\mathfrak{Ub}/J'$:
\begin{proposition}\label{prop:ugj-ubjprime}
    \begin{enumerate}
        \item $\mathfrak{Ub}/J'$ is isomorphic to $v_\lambda \otimes V_w(\mu)$ as $\mathfrak{Ub}$-modules.
        \item There is a surjective $\mathfrak{Ug}$-module map $\pi: \mathfrak{Ug}/J \to K(\lambda,w,\mu)$.
        \item The natural map $\phi: \mathfrak{Ub}/J' \to \mathfrak{Ug}/J$ is an injection and thus $J' = J \cap \mathfrak{Ub}$.
        \item Let $\nu \in P^+$. The restriction map
        \begin{equation}\label{eq:resmapcap}
\Res: \Hom_{\,\mathfrak{Ug}} (\mathfrak{Ug}/J, V(\nu)) \to \Hom_{\,\mathfrak{Ub}} (\mathfrak{Ub}/J', V(\nu))            
        \end{equation}
        is an isomorphism (of vector spaces).
    \end{enumerate}
\end{proposition}
\begin{proof}
Consider the following commutative diagram, where all indicated maps are the natural homomorphisms between the pairs of modules; while $\pi$ is $\mathfrak{Ug}$-linear, the other three are $\mathfrak{Ub}$-linear.
$$ \begin{tikzcd}
\mathfrak{Ub}/J' \arrow[r, "\phi"] \arrow[d, "\psi" ]
& \mathfrak{Ug}/J \arrow[d, two heads, "\pi"] \\
v_\lambda \otimes V_w(\mu) \arrow[r, hook, "\eta"]
&  K(\lambda,w,\mu)
\end{tikzcd}
$$   
The one-dimensional module $\mathbb{C} v_\lambda$ is acted upon trivially by $\mathfrak{n}^+$ and by $h - \langle \lambda, h \rangle 1$ for all $h \in \mathfrak{h}$. This, together with the presentation for the Demazure module $V_w(\mu)$ (Proposition~\ref{prop:dempres}) readily implies that $\psi$ is an isomorphism. 

The cyclic generator $v_\lambda \otimes v_{w\mu}$ of $K(\lambda,w,\mu)$ is annihilated by all the generators of the left ideal $J$. This is clear for the generators of types 1, 3, 4 (in the statement of Theorem~\ref{thm:gen-rel}). This fact, together with Lemma~\ref{lem:sl2string}, implies that the generators of type 2 also annihilate. Thus we obtain that the map $\pi: \mathfrak{Ug}/J \to K(\lambda,w,\mu)$ which sends $1 \to v_\lambda \otimes v_{w\mu}$ is a surjection. 
Finally, the map $\eta$ is the natural inclusion, and thus injective. 

The commutativity of the diagram and the properties of $\eta, \psi$ imply that $\phi$ is injective. This only leaves the last assertion of Proposition~\ref{prop:ugj-ubjprime} requiring proof.
An element $f \in \Hom_{\,\mathfrak{Ub}} (\mathfrak{Ub}/J', V(\nu))$ is uniquely determined by the image $f(\overline{1}) \in V(\nu)$ of the cyclic generator of $\mathfrak{Ub}/J'$. This element is annihilated by the generators of $J'$, and therefore by Lemma~\ref{lem:sl2string}, also by the generators of $J$. Thus $\overline{1} \mapsto f(1)$ also defines a homomorphism in $\Hom_{\,\mathfrak{Ug}} (\mathfrak{Ug}/J, V(\nu))$. This proves surjectivity of $\Res$. The injectivity of $\Res$ follows from the fact that the image of $\phi$ generates $\mathfrak{Ug}/J$. 
Thus $\Res$ is an isomorphism.
\end{proof}

\subsection{}\label{sec:proofends}
We are now ready to put together the ingredients developed so far to prove Theorem~\ref{thm:gen-rel}. We freely use the notation of the preceding subsections. It is clear from the proof of Proposition~\ref{prop:ugj-ubjprime} that we have a commutative diagram (of vector spaces): 
 $$ \begin{tikzcd}
\Hom_{\,\mathfrak{Ub}}(\mathfrak{Ub}/J', V(\nu)) \arrow[r, leftarrow, "\phi^*"] \arrow[d, leftarrow, "\psi^*" ]
& \Hom_{\,\mathfrak{Ug}}(\mathfrak{Ug}/J, V(\nu)) \arrow[d, leftarrow, "\pi^*"] \\
\Hom_{\,\mathfrak{Ub}}(v_\lambda \otimes V_w(\mu), V(\nu)) \arrow[r, leftarrow, "\eta^*"]
&  \Hom_{\,\mathfrak{Ug}}(K(\lambda,w,\mu), V(\nu))
\end{tikzcd}
$$ 
Since $\psi$ was an isomorphism, so is $\psi^*$. Observe also that the map $\phi^*$ is exactly the map {\em Res} of \eqref{eq:resmapcap} and the map $\eta^*$ is exactly {\em res} of \eqref{eq:resmap}. By Proposition~\ref{prop:ugj-ubjprime} and Theorem~\ref{thm:UbUg}, both {\em Res} and {\em res} are isomorphisms. We conclude that $\pi^*$ must also be an isomorphism. In particular,
\begin{equation}\label{eq:dimhomeq}
\dim \Hom_{\,\mathfrak{Ug}}(K(\lambda,w,\mu), V(\nu)) = \dim \Hom_{\,\mathfrak{Ug}}(\mathfrak{Ug}/J, V(\nu))
\end{equation}
for all $\nu \in P^+$. By Proposition~\ref{prop:catoint}, $\mathfrak{Ug}/J$ is in $\catoint$ and therefore a direct sum of irreducible highest weight $\mathfrak{g}$-modules. Equation~\eqref{eq:dimhomeq} implies that each irreducible module $V(\nu)$ occurs with the same multiplicity in $\mathfrak{Ug}/J$ and $K(\lambda,w,\mu)$. Since $\mathfrak{Ug}/J$ surjects onto $K(\lambda,w,\mu)$, they must in fact be isomorphic $\mathfrak{Ug}$-modules. This completes the proof of Theorem~\ref{thm:gen-rel}. \qed

\section{The tensor envelope of Kostant--Kumar modules}\label{sec:tensenvelope}
In this section, we restrict attention to the case when $\mathfrak{g}$ is a finite-dimensional semisimple Lie algebra.
\subsection{}
We recall the definition of the left ideal $J$ from Theorem~\ref{thm:gen-rel}.
We let $J_s \subset J$ denote the left ideal of $\mathfrak{Ug}$ generated by
\begin{enumerate}
        \item $e_\alpha^{\,p_\alpha (w\mu)+1}, \;\; \alpha$ is a {\em simple root}.
        \item $f_\alpha^{\,\langle \lambda, \alpha^\vee \rangle + q_\alpha (w\mu)+1}, \;\; \alpha$ is a {\em simple root}. 
        \item $ h - \langle \lambda + w\mu, h \rangle 1, \;\; h \in \mathfrak{h}$.
\end{enumerate}
In other words, we only take those generators of $J$ corresponding to the simple roots. Note that  since $\mathfrak{g}$ is finite-dimensional, there are no imaginary roots.

\begin{defn} Let $\mathfrak{g}$ be finite-dimensional semisimple. Let $\lambda, \mu \in P^+$ and $w \in W$.
The {\em tensor envelope}  $T(\lambda, w, \mu)$ of the Kostant--Kumar module $K(\lambda,w,\mu)$ is defined to be the $\mathfrak{Ug}$-module $$T(\lambda,w,\mu) := \mathfrak{Ug}/J_s.$$
\end{defn}

In order to describe the properties of the tensor envelope, we introduce the following notation. Let $\Lambda_i, \, i=1, \ldots, n$ denote the fundamental weights of $\lie g$. Given $\gamma = \sum_{i=1}^n c_i \Lambda_i \in P$, let its positive and negative parts $\gamma^{\pm} \in P^+$ be defined by:
\begin{equation}\label{eq:gammapm}
     \gamma^+ :=\sum_{i: \, c_i >0} c_i \Lambda_i \;\;\;\text{ and }\;\;\;
\gamma^- = -\sum_{i: \, c_i <0} c_i \Lambda_i.
\end{equation}

\begin{lemma}\label{lem:gammapmlem} Let $\gamma \in P$. Then
\begin{enumerate}
    \item $\gamma = \gamma^+ - \gamma^-$.
    \item $q_\alpha(\gamma) \leq \langle \gamma^+, \alpha^\vee \rangle  $ and $p_\alpha(\gamma) \leq \langle \gamma^-, \alpha^\vee \rangle$ for all positive roots $\alpha$.
    \item If $\alpha$ is a {\em simple root}, we have equality: $q_\alpha(\gamma) = \langle \gamma^+, \alpha^\vee \rangle  $ and $p_\alpha(\gamma) = \langle \gamma^-, \alpha^\vee \rangle$.
\end{enumerate}    
\end{lemma}
\begin{proof}
This follows from \eqref{eq:gammapm} and the definitions of $p_\alpha, q_\alpha$ in Equations~\eqref{eq:palpha}, \eqref{eq:qalpha}.
\end{proof}

Given $\zeta \in P^+$, we also find it convenient to let $\zeta^* = -w_0(\zeta)$ denote the highest weight of the dual representation $V(\zeta)^*$.
\begin{proposition}\label{prop:tensenv}
Let $\lie g$ be a finite-dimensional semisimple Lie algebra. Fix $\lambda, \mu \in P^+$ and let $w \in W$.
    \begin{enumerate}
        \item There is a surjective $\mathfrak{Ug}$-linear map $\phi_w: T(\lambda,w,\mu) \twoheadrightarrow K(\lambda,w,\mu)$. The map $\phi_w$ is an isomorphism for the two extreme cases $w=1$ and $w=w_0$.
        \item $T(\lambda,w,\mu) \cong V(\lambda + (w\mu)^+) \otimes V((w\mu)^-)^*$.
        \item $c_{\lambda\mu}^\nu(w) \leq c_{\lambda + (w\mu)^+, \, ((w\mu)^-)^*\;}^\nu\;$ for all $\nu \in P^+$.
    \end{enumerate}
\end{proposition}
\begin{proof}
Since $J_s \subset J$, there is a surjective $\mathfrak{Ug}$-linear map $\lie{Ug}/J_s \twoheadrightarrow \lie{Ug}/J \cong K(\lambda,w,\mu)$. We let $\phi_w$ denote this map; clearly $\phi_w(1 + J_s) = v_\lambda \otimes v_{w\mu}$. It follows from Remarks \ref{rem:hwtreprel} and \ref{rem:prvrel} that $J_s = J$ when $w=1$ or $w=w_0$. Hence $\phi_w$ is an isomorphism for $w=1, w_0$. This establishes the first part.

Let $\gamma_1 = \lambda + (w\mu)^+$ and $\gamma_2 = -w_0((w\mu)^-))$, i.e., $V(\gamma_2) = V((w\mu)^-)^*$. Let $\alpha$ be a {\em simple root}. Then Lemma~\ref{lem:gammapmlem} implies that: (i) $\langle \gamma_1, \alpha^\vee \rangle = \langle \lambda, \alpha^\vee \rangle + q_\alpha(w\mu)$ and 
\begin{equation}\label{eq:partwo}
    \mathrm{(ii)}\; p_\alpha(w_0\gamma_2) = - \langle w_0\gamma_2, \alpha^\vee \rangle = \langle (w\mu)^-, \alpha^\vee \rangle = p_\alpha(w\mu).
\end{equation}
We also have using \eqref{eq:gammapm} that (iii) $\gamma_1 + w_0 \gamma_2 = \lambda + w\mu$.

Observations (i)-(iii) above together with Remark~\ref{rem:prvrel} imply that 
\[\lie{Ug}/J_s \cong K(\gamma_1, w_0, \gamma_2) \cong V(\lambda + (w\mu)^+) \otimes V((w\mu)^-)^*,\]
establishing the second part. We note here that the image of the coset $1+J_s \in \lie{Ug}/J_s$ under this isomorphism is exactly $v_{\gamma_1} \otimes v_{w_0 \gamma_2}$, i.e., 
\begin{equation}\label{eq:tensenv-canmor}
    \phi_w: V(\gamma_1) \otimes V(\gamma_2) \twoheadrightarrow K(\lambda,w,\mu)  \text{ maps }  v_{\gamma_1} \otimes v_{w_0 \gamma_2} \mapsto v_{\lambda} \otimes v_{w\mu}. 
\end{equation}
Finally, the third part of the Proposition is a simple consequence of the second.
\end{proof}

\begin{rem}
Since $K(\lambda,w,\mu) \subseteq V(\lambda) \otimes V(\mu)$, we have the naive upper bound $\lrw \leq \lr$ for all $w \in W$. The third part of Proposition~\ref{prop:tensenv} gives a $w$-dependent upper bound in terms of tensor product multiplicities, which is often tighter in practice. For instance, when $\lie g = \lie{sl}_3$, 
$\lambda =  \nu = \Lambda_1 + 2\Lambda_2, \,\mu = \Lambda_1 + \Lambda_2$ and $w=s_2s_1$, we have $\lrw = 1$ and $\lr = 2$. The tensor envelope is $V(2\Lambda_1 + 2\Lambda_2) \otimes V(2\Lambda_2)^*$ and it can be easily verified that the multiplicity of $V(\nu)$ in this module is $1$. 
\end{rem}

\begin{rem}
The tensor envelope $T$ of a Kostant--Kumar module $K$ has an interesting universal property - among all tensor products that admit surjections to $K$ (of a special kind), $T$ is the "smallest". This will be discussed in Section \ref{sec:tensenv-univprop} below.
\end{rem}

\begin{rem}
    Gould and Edwards \cite[Theorem 4]{GE} have a more general construction of tensor envelopes (though not named thus) for all cyclic submodules of finite-dimensional $\lie{g}$-modules. Our construction may be viewed as a special case of theirs.
\end{rem}
\section{Schur positivity and canonical morphisms}\label{sec:canmors}
\subsection{}
Let $\lie g$ be a symmetrizable Kac--Moody algebra. If $V$ and $V'$ are modules in $\catoint$, then we say that $V$ is {\em Schur positive} relative to $V'$ (or that $V \ominus V'$ is Schur positive) if there exists a $\mathfrak{Ug}$-linear surjection $V \twoheadrightarrow V'$, or equivalently a $\lie{Ug}$-linear injection $V' \hookrightarrow V$. In finite type $A$, this coincides with the usual notion - namely that the difference of characters of $V$ and $V'$ is a nonnegative integral linear combination of Schur functions.

In type $A$, the Schur positivity of tensor products $V(\lambda_1)\otimes V(\mu_1) \twoheadrightarrow V(\lambda_2)\otimes V(\mu_2)$ was extensively studied in \cite{FFLP, LPP, DP, CFS} under various hypotheses on the highest weights involved.

\subsection{}
Since Kostant--Kumar modules are generalizations of tensor products, Schur positivity is a natural question in the setting of Kostant--Kumar modules as well. 
The following proposition gives a necessary condition for Schur positivity.

\begin{proposition}\label{prop:schurpos-kk-nec}
Let $\lie g$ be a symmetrizable Kac--Moody algebra.    If  $\; K(\lambda_1, w_1, \mu_1) \ominus K(\lambda_2, w_2, \mu_2)\;$ is Schur positive, then \[\lambda_1 + \mu_1 \,\geq\, \lambda_2 + \mu_2 \,\geq\,\overline{\lambda_2 + w_2\mu_2} \,\geq\, \overline{\lambda_1 + w_1\mu_1}.\] 
\end{proposition}
\begin{proof}
Let $K_1:=K(\lambda_1, w_1, \mu_1)$ and $K_2:=K(\lambda_2, w_2, \mu_2)$. Since $K_2$ is a submodule of $V(\lambda_2) \otimes V(\mu_2)$, all weights of $K_2$ are $\leq \lambda_2 + \mu_2$. Thus $\lambda_2 + \mu_2 \geq \overline{\lambda_2 + w_2\mu_2}$ since the latter is a weight of $K_2$. It remains to establish the other inequalities.
The Schur positivity hypothesis implies that there is a $\lie{Ug}$-linear injection $K_2:=K(\lambda_2, w_2, \mu_2) \hookrightarrow K(\lambda_1, w_1, \mu_1)=:K_1$. 
The second assertion of Lemma~\ref{lem:supplrw} implies that $V(\lambda_2+\mu_2)$ and $V(\overline{\lambda_2 + w_2\mu_2})$ occur as direct summands of $K_2$; hence they also occur as direct summands of $K_1$. Another application of Lemma~\ref{lem:supplrw} (this time to $K_1$) leads to the desired conclusion. 
\end{proof}

\subsection{}
To find sufficient conditions under which Schur positivity holds for $K(\lambda_1, w_1, \mu_1) \ominus K(\lambda_2, w_2, \mu_2)$, we look for the existence of a $\lie{Ug}$-linear surjection $K(\lambda_1, w_1, \mu_1) \twoheadrightarrow K(\lambda_2, w_2, \mu_2)$. We will in fact construct special surjections ("{\em canonical morphisms}") below under appropriate conditions.

Let $\lie g$ be a symmetrizable Kac--Moody algebra.  For $i=1,2$, let $v_{\lambda_i} \otimes v_{w_i \mu_i}$ denote the cyclic generator of the Kostant--Kumar module $K(\lambda_i,w_i,\mu_i)$  and let its annihilator be $J_i = \ann_{\lie{Ug}}(v_{\lambda_i} \otimes v_{w_i\mu_i})$. The map
\begin{align*} 
K(\lambda_1,w_1,\mu_1) &\to K(\lambda_2,w_2,\mu_2) \\
v_{\lambda_1} \otimes v_{w_1\mu_1} &\mapsto v_{\lambda_2} \otimes v_{w_2\mu_2}
\end{align*}
extends to a well-defined $\lie{Ug}$-linear map on all of $K(\lambda_1,w_1,\mu_1)$ provided $J_1 \subseteq J_2$. When this map exists, we call it the {\em canonical morphism} between these Kostant--Kumar modules. It is uniquely defined (when it exists) up to a nonzero scalar, since the cyclic generators above are only well-defined up to scaling.  
A composition of canonical morphisms is clearly a canonical morphism.
We observe that two Kostant--Kumar modules may have several $\lie{Ug}$-linear morphisms between them, but may not admit a canonical morphism.

Since a canonical morphism is clearly surjective, existence of a canonical morphism 
$K(\lambda_1,w_1,\mu_1) \to K(\lambda_2,w_2,\mu_2)$ implies Schur positivity of 
$K(\lambda_1,w_1,\mu_1) \ominus K(\lambda_2,w_2,\mu_2)$.

\subsection{} If a map $\phi: K(\lambda_1,w_1,\mu_1) \to K(\lambda_2,w_2,\mu_2)$ is a canonical morphism, we depict this fact by the notation \[\phi: K(\lambda_1,w_1,\mu_1) \canmor K(\lambda_2,w_2,\mu_2).\]
\begin{proposition}\label{prop:canmor}
    Let $\lie g$ be a finite-dimensional semisimple Lie algebra or a symmetric Kac--Moody algebra. Let $\lambda_i, \mu_i \in P^+$ and $w_i \in W$ for $i=1, 2$. The following are equivalent:
    \begin{enumerate}
        \item[(i)] There exists a canonical morphism $K(\lambda_1,w_1,\mu_1) \canmor K(\lambda_2,w_2,\mu_2)$.
        \item[(ii)] $\lambda_1 + w_1\mu_1 = \lambda_2 + w_2\mu_2$ and $p_\alpha(w_1\mu_1) \geq p_\alpha(w_2\mu_2)$ for all $\alpha \in \posreal$.
        \item[(iii)] $\lambda_1 + w_1\mu_1 = \lambda_2 + w_2\mu_2$ and $\langle \lambda_1 - \lambda_2, \alpha^\vee \rangle \geq 0$ for all $\alpha \in \posreal$ such that \\ $\langle w_2\mu_2, \alpha^\vee\rangle <0$.
    \end{enumerate}
\end{proposition}
\begin{proof}
Since a canonical morphism must map $v_{\lambda_1} \otimes v_{w_1\mu_1} \mapsto v_{\lambda_2} \otimes v_{w_2\mu_2}$, it follows from the relations in Theorem~\ref{thm:gen-rel} that (i) $\Rightarrow$ (ii). 

To show (ii) $\Rightarrow$ (i), we consider the map sending $v_{\lambda_1} \otimes v_{w_1\mu_1} \mapsto v_{\lambda_2} \otimes v_{w_2\mu_2}$. To verify this is well-defined, we need to verify that any element $u \in \lie{Ug}$ which annihilates $v_{\lambda_1} \otimes v_{w_1\mu_1}$ also annihilates $v_{\lambda_2} \otimes v_{w_2\mu_2}$. We can assume that $u$ is an element of one of the forms (1)--(4) of Theorem~\ref{thm:gen-rel}. When $u$ is of form (1) or (4), the desired conclusion follows from our hypothesis. When $u$ is an element of the form (2), this follows from Lemma~\ref{lem:sl2string}, while the conclusion is trivial for $u$ of the form (3).

The equivalence of (ii) and  (iii) is a straightforward consequence of the definition of $p_\alpha$ in Equation~\eqref{eq:palpha}.
\end{proof}

\begin{example}
We consider the affine Lie algebra $\mathfrak{g}=\widehat{\lie{sl}_2}$. Let $\alpha_0, \alpha_1$ denote its simple roots and let $\delta=\alpha_0 + \alpha_1$ be its null root. There exists a canonical morphism 
\[ K(\Lambda_0 + \Lambda_1, \,s_0 s_1 s_0, \,\Lambda_0 + 2\delta) \canmor 
K(\Lambda_1, s_0, 2\Lambda_0) \]
as can be easily verified from Proposition~\ref{prop:canmor}.
\end{example}

\subsection{}
We readily recover from Proposition~\ref{prop:canmor} the following Schur positivity result due to Mathieu \cite[Lemma 2.1]{M2} in the finite-dimensional case:
\begin{proposition}\label{lem:mathieu} (Mathieu)
Let $\lie g$ be a finite-dimensional semisimple Lie algebra and $\lambda_1, \lambda_2, \mu_1, \mu_2 \in P^+$ be such that $\lambda_1 - \lambda_2 = w_0(\mu_2 - \mu_1) \in P^+$. Then there is a surjective map of $\lie{Ug}$-modules $V(\lambda_1) \otimes V(\mu_1) \twoheadrightarrow V(\lambda_2) \otimes V(\mu_2)$, i.e., \[[V(\lambda_1) \otimes V(\mu_1)] \ominus [V(\lambda_2) \otimes V(\mu_2)]\] is Schur positive.
\end{proposition}
\begin{proof}
The hypothesis ensures $\lambda_1 + w_0(\mu_1) = \lambda_2 + w_0(\mu_2)$ and that $\lambda_1 - \lambda_2 \in P^+$. The condition (iii) of Proposition~\ref{prop:canmor} ensures that there is a canonical morphism $K(\lambda_1,w_0,\mu_1) \canmor K(\lambda_2,w_0,\mu_2)$. This is in particular a surjective $\lie{Ug}$-morphism.
\end{proof}

\subsection{}
 Given $w \in W$, we let \[I(w) :=\{ \alpha\in \posreal:  w^{-1}\alpha \in \posreal[-]\}\]
 denote its set of "inversions". We recall that $|I(w)|$ is the length of $w$. Define
\[ P_w^+ :=\{\lambda \in P: \langle \lambda, \alpha^\vee \rangle \geq 0 \text{ for all } \alpha\in I(w)\}.\]
For $\lie g$ finite-dimensional, we observe that $P_{w_0}^+ = P^+$. 
An interesting extension of Mathieu's Schur positivity result (Proposition~\ref{lem:mathieu}) to the general case is the following:
\begin{proposition}
    \label{lem:mathieu-infinite}
Let $\lie g$ be a finite-dimensional semisimple Lie algebra or a symmetric Kac--Moody algebra. Let $\lambda_1, \lambda_2, \mu_1, \mu_2 \in P^+$ and $w \in W$ be such that $\lambda_1 - \lambda_2 = w(\mu_2 - \mu_1) \in P_w^+$. Then there is a canonical morphism  $K(\lambda_1, w, \mu_1) \canmor K(\lambda_2, w, \mu_2)$ (and hence a surjective map of $\lie{Ug}$-modules). In particular, 
$K(\lambda_1,w,\mu_1) \ominus K(\lambda_2, w, \mu_2)$  is Schur positive.
\end{proposition}
\begin{proof}
Suppose $\alpha \in \posreal$ such that $\langle w\mu_2, \alpha^\vee\rangle <0$. Since $\langle w\mu_2, \alpha^\vee\rangle = \langle \mu_2, w^{-1}(\alpha^\vee) \rangle$, we conclude that $\alpha \in I(w)$. For such $\alpha$, our hypothesis ensures $\langle \lambda_1 - \lambda_2, \alpha^\vee \rangle \geq 0$. The desired conclusion now follows from Proposition~\ref{prop:canmor} (iii).
\end{proof}

The following result establishes Schur positivity of differences of Kostant--Kumar modules under more general conditions. It can be viewed as a $\sigma$-twisted version of Proposition~\ref{prop:canmor} for $\sigma \in W$, using the generators of the form $v_{\sigma \lambda} \otimes v_{\sigma w\mu}$ of $K(\lambda,w,\mu)$ provided by Proposition~\ref{prop:kkprops}(2). We leave the straightforward details of the proof to the interested reader (see also \cite{manika-thesis}).

\begin{proposition}\label{prop:sigma}
 Fix $\sigma \in W$. Let $\lambda_1, \lambda_2, \mu_1, \mu_2 \in P^+$ and $w_1, w_2 \in W$. There is a surjective $\lie{Ug}$-linear map $K(\lambda_1,w_1,\mu_1) \twoheadrightarrow K(\lambda_2,w_2,\mu_2)$ if all the following conditions hold: 
    \begin{itemize}
        \item $\lambda_1+w_1\mu_1=\sigma\lambda_2+\sigma w_2\mu_2$
        \item  $p_\alpha(w_1\mu_1)\geq p_\alpha(\sigma w_2\mu_2)$ for  $\alpha\in \Delta_{re}^+$ such that $\sigma^{-1}\alpha\in \Delta_{re}^+$.
        \item $p_\alpha(w_1\mu_1) \geq  q_\alpha(\sigma w_2\mu_2) + \langle\sigma \lambda_2,\alpha^\vee\rangle$ for $\alpha\in \Delta_{re}^+$ such that $\sigma^{-1}\alpha\in \Delta_{re}^-$.  \qed
            \end{itemize} 
            This map sends $v_{\lambda_1} \otimes v_{w_1\mu_1} \mapsto v_{\sigma\lambda_2} \otimes v_{\sigma w_2\mu_2}.$
\end{proposition}

\begin{example}
Let $\lie g = \lie{sl}_4$ and let $\Lambda_i$ denote its fundamental weights and $s_i \in W$
be the simple reflections for $i=1, 2, 3$. Then there exists a surjective $\lie{Ug}$-linear map $K(3\Lambda_1+2\Lambda_3, s_2s_1,3\Lambda_1) \twoheadrightarrow K(\rho, s_1, \rho)$ where $\rho = \Lambda_1 + \Lambda_2 + \Lambda_3$ is the Weyl vector. This follows from Proposition~\ref{prop:sigma} by choosing $\sigma = s_2$. 
\end{example} 
 
\subsection{}\label{sec:tensenv-univprop}
In this subsection, we assume that $\lie g$ is a finite dimensional semisimple Lie algebra. We recall the discussion of tensor envelopes and the map $\phi_w: T(\lambda,w,\mu) \twoheadrightarrow K(\lambda,w,\mu)$ from Proposition~\ref{prop:tensenv}  It follows from the proof of Proposition~\ref{prop:tensenv} (Equation \eqref{eq:tensenv-canmor}) that $\phi_w$ is in fact a canonical morphism.
The tensor envelope has the following universal property.
\begin{proposition}\label{prop:tensenv-univprop}
  Let $\lie g$ be a finite-dimensional semisimple Lie algebra. Let $\lambda, \mu \in P^+$ and $w \in W$. Let $\lambda', \mu' \in P^+$. If $\phi: V(\lambda') \otimes V(\mu') \canmor K(\lambda,w,\mu)$ is a canonical morphism, then there exists a unique canonical morphism $\widetilde{\phi}: V(\lambda') \otimes V(\mu') \canmor T(\lambda,w,\mu)$ such that the diagram commutes:
\[\begin{tikzcd}
V(\lambda') \otimes V(\mu') \arrow[d, bulletepi, "\widetilde{\phi}"'] \arrow[r, bulletepi, "\phi"] & K(\lambda,w,\mu) \\
T(\lambda,w,\mu)  \arrow[ur, bulletepi, "\phi_w"']
\end{tikzcd}\]
\end{proposition}
\begin{proof}
   We write $\gamma_1 = \lambda + (w\mu)^+$ and $\gamma_2 = -w_0((w\mu)^-)$ as in the proof of Proposition~\ref{prop:tensenv}. Thus $T(\lambda,w,\mu) = K(\gamma_1,w_0,\gamma_2)$ and $\gamma_1 + w_0 \gamma_2 = \lambda+w\mu$.  

 Since  $V(\lambda') \otimes V(\mu') = K(\lambda',w_0,\mu')$, it follows from Proposition~\ref{prop:canmor} that the existence of the canonical morphism $\phi$ ensures $\lambda'+ w_0\mu' = \lambda+w\mu$ and $p_\alpha(w_0\mu') \geq p_\alpha(w\mu)$ for all positive roots $\alpha$. If we now further specialize $\alpha$ to a simple root, we have from \eqref{eq:partwo} that $p_\alpha(w_0\gamma_2) = p_\alpha(w\mu)$. In conclusion, $p_\alpha(w_0\mu') \geq p_\alpha(w_0\gamma_2)$ for all simple roots $\alpha$, i.e., $- \langle w_0\mu', \alpha^\vee \rangle + \langle w_0\gamma_2, \alpha^\vee \rangle \geq 0$. Since this inequality holds for all simple roots $\alpha$, it also holds for all positive roots by linearity. The existence of $\widetilde{\phi}$ now follows from Proposition~\ref{prop:canmor}. Its uniqueness is obvious.
\end{proof}

It readily follows from this proposition that a canonical morphism of Kostant--Kumar modules induces a canonical morphism of their tensor envelopes. We leave the proof of this statement below to the interested reader.
\begin{proposition}\label{prop:tensenv-canmors}
Let $\lie g$ be a finite-dimensional semisimple Lie algebra. Let $\lambda, \mu, \lambda', \mu' \in P^+$ and $w, w' \in W$. If $\psi: K(\lambda',w',\mu') \canmor K(\lambda,w,\mu)$ is a canonical morphism, then there exists a unique canonical morphism $\widetilde{\psi}: T(\lambda',w',\mu') \canmor T(\lambda,w,\mu)$ such that the diagram commutes:
\[\begin{tikzcd}
T(\lambda',w',\mu') \arrow[d, bulletepi, "\phi'_w"'] \arrow[r, bulletepi, "\widetilde{\psi}"] & T(\lambda,w,\mu) \arrow[d, bulletepi, "\phi_w"] \\
K(\lambda',w',\mu')  \arrow[r, bulletepi, "\psi"'] & K(\lambda,w,\mu)
\end{tikzcd}\]\qed
\end{proposition}

\begin{rem}
One can likewise define the notion of canonical isomorphism. It follows from Proposition~\ref{prop:tensenv-canmors} that a canonical isomorphism of Kostant--Kumar modules induces a canonical isomorphism of their tensor envelopes. However the converse fails; consider for $\lie g = \lie{sl}_3$: $K(0,s_2, \Lambda_1 + \Lambda_2) \not\cong K(2\Lambda_1, s_2s_1, \Lambda_1)$ even as $\lie{Ug}$-modules, while their tensor envelopes are both canonically isomorphic to $V(2\Lambda_1) \otimes V(\Lambda_1)$.
\end{rem}

\printbibliography

\end{document}